\documentclass[a4paper,11pt]{article}
\usepackage{amsfonts,amssymb,amsbsy,amsmath,epsf,eepic,graphics,graphicx}
\usepackage{color}
\usepackage{ mathrsfs }
\usepackage{bbm}
\usepackage{enumerate}
\usepackage[english]{babel}
\usepackage[hidelinks]{hyperref}
\makeatletter
\@addtoreset{enunciato}{section}
\makeatother

\newcounter{enunciato}[section]

\makeatletter
\@addtoreset{enunciat}{subsection}
\makeatother

\newcounter{enunciat}[subsection]

\newtheorem{ittheorem}{Theorem}
\newtheorem{ittheoremm}{Theorem}
\newtheorem{itlemma}{Lemma}
\newtheorem{itlemmaa}{Lemma}
\newtheorem{itproposition}{Proposition}
\newtheorem{itpropositiona}{Proposition}
\newtheorem{itdefinition}{Definition}
\newtheorem{itconjecture}{Conjecture}
\newtheorem{itassumption}{Assumption}
\newtheorem{itcorollary}{Corollary}
\newtheorem{itcorollarya}{Corollary}
\newtheorem{itremark}{Remark}

\newenvironment{remark}{\addtocounter{enunciato}{1}   	\begin{itremark}}{\end{itremark}}

\newenvironment{theorem}{\addtocounter{enunciato}{1}
	\begin{ittheorem}}{\end{ittheorem}}

\newenvironment{lemma}{\addtocounter{enunciato}{1}
	\begin{itlemma}}{\end{itlemma}}

\newenvironment{proposition}{\addtocounter{enunciato}{1}
	\begin{itproposition}}{\end{itproposition}}

\newenvironment{definition}{\addtocounter{enunciato}{1}
	\begin{itdefinition}}{\end{itdefinition}}

\newenvironment{assumption}{\addtocounter{enunciato}{1}
	\begin{itassumption}}{\end{itassumption}}

\newenvironment{corollary}{\addtocounter{enunciato}{1}
	\begin{itcorollary}}{\end{itcorollary}}

\newenvironment{proof}{\noindent {\em Proof}.\,\,}
{\hspace*{\fill}$\square$\medskip}

\newenvironment{proofofp}{\noindent {\em Proof of Proposition\,\,}}
{\hspace*{\fill}$\square$\medskip}
\newenvironment{proofoft}{\noindent {\em Proof of Theorem\,\,}}
{\hspace*{\fill}$\square$\medskip}

\newcommand\restr[2]{{
		\left.\kern-\nulldelimiterspace 
		#1 
		\vphantom{\big|} 
		\right|_{#2} 
}}

\newcommand\smallrestr[2]{{
		\left.\kern-\nulldelimiterspace 
		#1 
		\vphantom{|} 
		\right|_{#2} 
}}

\def \R {{\mathbb R}}
\def \N {{\mathbb N}}

\def \ra {\rightarrow}

\def \barr{\begin{array}}
	\def \earr {\end{array}}

\def \cH {{\mathcal H}}

\def \cC {{\mathcal C}}

\def \cP {{\mathcal P}}

\def \cA {{\mathcal A}}

\def \cF {{\mathcal F}}

\newcommand{\Pw}{\mathcal{P}_2(\R)}

\newcommand{\qaq}{\quad \text{ and } \quad}
\newcommand{\qf}{\quad \text{ for }  }
\newcommand{\qfa}{\quad \text{ for all }  }
 
\definecolor{darkgreen}{rgb}{0,.6,0}
\definecolor{darkagenta}{rgb}{.5,0,.5}
\definecolor{darkred}{rgb}{1,0,0}
\definecolor{darkblue}{rgb}{0,0,.4}
\definecolor{black}{rgb}{0,0,0}
\definecolor{gray}{rgb}{.4,.4,.4}
\definecolor{white}{rgb}{0.99,0.99,0.99}
\definecolor{geel1}{rgb}{1,1,0.3}

\newcommand{\mr}{\mathrm{m}}

\numberwithin{equation}{section}

\newcommand{\eee}{\mathrm{e}}

\newcommand{\ACLio}{\cA\cC((0,\infty);\Pw)}

\newcommand\CouplW[1]{\mathrm{Cpl}(#1)}

\newcommand\Lrm {\mathrm{L}} 

\begin{document}
	\bibliographystyle{abbrv}
\title{On the long-time behaviour of McKean-Vlasov paths
}

\author{
	\renewcommand{\thefootnote}{\arabic{footnote}}
	K.\ Bashiri\footnotemark[1]  
}

\footnotetext[1]{
	Institut f\"ur Angewandte Mathematik,
	Rheinische Friedrich-Wilhelms-Universit\"at, 
	Endenicher Allee 60, 53115 Bonn, Germany.
	Email:
	bashiri@iam.uni-bonn.de.
}

 \vspace{-3cm}
\maketitle
 
 \vspace{-1cm}
 \begin{abstract}
 It is well-known that, in a certain parameter regime, the so-called \emph{McKean-Vlasov evolution} $ (\mu_t)_{t\in [0,\infty)} $ admits exactly three \emph{stationary states}.
 In this   paper we study the \emph{long-time behaviour} of the flow $ (\mu_t)_{t\in [0,\infty)} $ in this regime. 
 The main result is that,  
 for any initial measure $ \mu_0 $, the flow $ (\mu_t)_{t\in [0,\infty)} $   converges to a stationary state as $ t\ra \infty $ (see Theorem \ref{ThmConv}).
Moreover, we show that if the energy of the  initial measure is below some critical threshold, then the limiting stationary state can be identified (see Proposition \ref{PropIntroSmallBasin}). 
Finally, we also show some topological properties of the basins of attraction of the McKean-Vlasov evolution  (see Proposition \ref{PropIntroBasin}). 
 The proofs are based on the representation of $ (\mu_t)_{t\in [0,\infty)} $   as a \emph{Wasserstein gradient flow}.
 
Some results  of this paper are not entirely new. 
The main contribution here is to  show that the \emph{Wasserstein framework} provides short and elegant proofs for these results.
  However, up to the author's best knowledge, the statement on the  topological properties of the basins of attraction (Proposition \ref{PropIntroBasin}) is a new result.

 \let\thefootnote\svthefootnote{  

}
\let\thefootnote\relax\footnotetext[0]{
K.B.\   is partially  supported by the Deutsche Forschungsgemeinschaft (DFG, German Research Foundation) - Projektnummer 211504053 - SFB 1060 and by Germany's Excellence Strategy -- GZ 2047/1, projekt-id 390685813 --  ``Hausdorff Center for Mathematics'' at Bonn University.}
\end{abstract}


\section{Introduction}
In this paper we study the \emph{ergodicity} and the \emph{energy landscape} of the  flow $ (\mu_t)_{t\in [0,\infty)} $ of  marginal laws associated to  the stochastic differential equation given by
\begin{align}\label{EqSDEBasin}
	dx_t ~=~ - \Psi'(x_t) \, dt + J \int_{\R} z \, d\mu_t(z) \, dt + \sqrt{2}\,dB_t.
\end{align}
Here,  
the \emph{single-site potential} $ \Psi:\R \ra \R $ and the \emph{interaction strength} $ J\in \R $ satisfy Assumption \ref{AssBasin} below, and $ B $ is a one-dimensional Brownian motion.
This flow $ (\mu_t)_{t\in [0,\infty)} $  is often called \emph{McKean-Vlasov evolution} in the literature.

In order to understand the main motivation for this paper, we recall five  well-known facts.
\begin{enumerate}[(i)]
	\setlength\itemsep{-0.1em} 
	\item Let $ (\Pw, W_2) $ be the \emph{Wasserstein space}; see Section \ref{SubsecWassSpace}  below. Then, we know from \cite[Chapter 11]{ambgigsav} and \cite{CarrilloMcCannVillani} that $ (\mu_t)_{t\in [0,\infty)} $ can be represented as a so-called \emph{Wasserstein gradient flow} (again see Section \ref{SubsecWassSpace})  for the functional  $ \cF:\Pw \ra (-\infty,\infty]  $, which is  defined by 
	\begin{align}\label{EqDefiF}
		\cF(\mu )~=~ 
		\int_{\R}  \log(\rho) d\mu +  \int_{\R}  \Psi \, d\mu  - \frac{J}{2} \left( \int_{\R} z \, d\mu(z) \right) ^2
	\end{align}
	if $ \mu \in \Pw $ has a Lebesgue density $ \rho $,  and $ \cF(\mu) = \infty $ otherwise.
	Moreover, in \cite[11.2.8]{ambgigsav} it is shown that   for all $ \mu \in \overline{D(\cF)}=\Pw $ (where $ D(\cF)= \{ \mu \in \Pw \,|\, \cF(\mu) < \infty\} $), there exists a unique
	Wasserstein gradient flow for $ \cF $ with initial value $ \mu $. In this paper, we denote this gradient flow   by $ (S[\mu](t) )_{t\in (0,\infty)}$.
	\item Consider  the  system of $ N \in \N $ \emph{mean-field interacting diffusions} given by
	\begin{align}\label{EqIntroMicrSystem}
		\begin{split}
			dx^{N}_i (t)&= - \Psi' \left(x^{N}_i (t)\right) \, dt + \frac{J}{N} \sum_{j=0	}^{N-1}  x_j^{N}(t)\, dt  + \sqrt{2  }\,  dB_i (t) \qf   0\leq i \leq N-1,
		\end{split}
	\end{align}  
	where  
	$ B^{N} = (B_i)_{i=0,\dots,N-1}  $ is an $ N $-dimensional Brownian motion.
	Let $ (L^N(t))_{t\in[0,\infty)} $ denote the corresponding 
	\emph{empirical distribution process}, i.e., 
	\begin{align}
		L^N(t)~=~\frac{1}{N}
		\sum_{i=0}^{N-1} \delta_{x^{N}_i (t)} \qfa t\in[0,\infty).
	\end{align}
	Then, ever since the   classic papers \cite{dawgae} and \cite{Gae}, it is known that the process $ (L^N(t))_{t\in[0,\infty)} $ converges weakly to the deterministic McKean-Vlasov evolution as $ N \ra \infty $. 	
	\item Already in the paper \cite{dawgaetunneling} it was conjectured that the process  $ (L^N(t))_{t\in[0,\infty)} $  exhibits \emph{metastable behaviour}\footnote{We refer the reader with no background in metastability to the monumental monographs in this subject given by \cite{BdH15} and \cite{OV04}.}.
	It is a long outstanding problem to verify this conjecture   rigorously. 
	Although some progress in this direction was established in the papers \cite{BaMe2019} and \cite{GvalaniSchlichting2020}, there are still many open and challenging questions.
	\item It is well-known   that, in order to analyse the metastable behaviour of a stochastic system, it is essential to have deep knowledge on the underlying \emph{energy landscape} of the system and its \emph{ergodicity}, i.e., its possible convergence towards stationary measures.  
	\item In order to study curves and other objects that belong to the infinite-dimensional \emph{space of probability measures},   the  \emph{Wasserstein formalism} provides a natural and convenient framework. 
	Indeed, ever since the seminal papers \cite{jko} and \cite{Ott01}, it is known that the  Wasserstein formalism  provides the structure of a \emph{Riemannian manifold} on the space of probability measures. 
	We refer to \cite[p.\ 421]{ArnrichMielkePeletierSavareVeneroni} or \cite[Section 1.4]{BaMe2019} for more arguments that speak in favour of the Wasserstein formalism.  
\end{enumerate}

We now formulate the main motivation for this paper. 
Combining the facts (ii), (iii) and (iv), we see that, in order to understand the metastable behaviour of $ (L^N(t))_{t\in[0,\infty)} $, it   is essential to study the \emph{ergodicity} and the \emph{energy landscape} of the  McKean-Vlasov evolution. 
Moreover, from fact (i) we see that the energy landscape   associated to $ (\mu_t)_{t\in [0,\infty)} $ is determined by the functional $ \cF $ and its \emph{basins of attraction}; see Proposition \ref{PropIntroBasin} for the precise definition of the latter.
This is the main motivation why we study the  ergodicity of $ (\mu_t)_{t\in [0,\infty)} $ and the basins of attraction of $ \cF $.
Finally, fact (v) explains why we use the Wasserstein setting as the framework for this paper. 


We make  the following assumptions  throughout this paper. 
\begin{assumption}\label{AssBasin}
	\begin{enumerate}[(1)]
		\setlength\itemsep{-0.1em} 
		\item There is a splitting $\Psi= \Psi_c + \Psi_b$ for some $ \Psi_c , \Psi_b \in C^2(\R) $, and there are constants~$ 0<c ,c' < \infty$ such that $ 		\Psi_c'' \geq c     \mbox{ and }   |\Psi_b|+|\Psi_b'|+|\Psi_b''| \leq c'  $ on $ \R $.
		\item There exist  $ \epsilon, c'' \in (0,\infty) $ such that  $\Psi(z) \geq c''( |z|^{2+\epsilon} -1)$ for all 
		$ z \in \R $.  
		\item $\Psi(z)=\Psi(-z)$ for all~$z \in \mathbb{R}$. 
		\item $ 1/J < \int_{\R} z^2 \, \eee^{-\Psi(z)  } \, dz  / (\int \eee^{-\Psi(z)  } \, dz)   $.  
		\item $ z \mapsto \Psi'(z) $ is  convex on $ [0,\infty) $.
	\end{enumerate}
\end{assumption}
In particular, Assumption \ref{AssBasin} is fulfilled if $ \Psi$  is   a polynomial of  degree $ 2 \ell $ for some $ \ell \in \N \cap [2,\infty) $  such that Assumption \ref{AssBasin} (4) and Assumption \ref{AssBasin} (5) are satisfied. 
In Section \ref{SecAssu} we briefly discuss  the assumptions we make in this paper.

An important  observation in Lemma \ref{LemStationary} is that, as an immediate consequence of Assumption \ref{AssBasin}, the system \eqref{EqSDEBasin}
admits exactly three \emph{stationary points} at some measures  $\mu^-,\mu^0, \mu^+ \in \Pw $, which  are defined in \eqref{EqBasinStationaryMeasures}; see Lemma \ref{LemStationary} for more details. 
We also mention here that, as we will see in Lemma \ref{LemRelationFH},  the measures  $ \mu^-  $ and  $  \mu^+ $ are the \emph{global minimizers} of the functional $ \cF $.

We   now  formulate the main result of this paper  in the following theorem. 
\begin{theorem}\label{ThmConv}
	Suppose Assumption \ref{AssBasin}.
	Let $ \mu \in \Pw $. 
	Then, there exists $ \mu^* \in \{\mu^-,\mu^0, \mu^+\} $ such that 
	\begin{align}
		\lim_{t\ra \infty } W_2(S[\mu](t),\mu^*) ~=~
		0 \qaq \lim_{t\ra \infty } \cF(S[\mu](t)) = \cF(\mu^* ).
	\end{align} 
\end{theorem}
\begin{proof}
	The proof is postponed to Section \ref{SecProofThmConv}.
\end{proof}

As a by-product of the proof of Theorem \ref{ThmConv}, we obtain the following two propositions, which are interesting on their own. 
The first one   shows that inside the valleys of the set $ \{\mu\in\Pw\,|\, \cF(\mu)\leq \cF(\mu^0)  \} $ the convergence of the gradient flows  for $ \cF $ is determined by the sign of the mean of the initial value. 
\begin{proposition}\label{PropIntroSmallBasin}
	Suppose Assumption \ref{AssBasin}.
	Let $ \mu \in \Pw $ be such that $ \int_\R  z\, d\mu(z) \neq 0 $ and $ \cF(\mu)\leq \cF(\mu^0)  $.
	Then, 
	\begin{align} \label{EqSmallBasin0}
		\lim_{t\ra \infty } \cF(S[\mu](t)) ~=~ \cF(\mu^- )~=~ \cF(\mu^+ ),  
	\end{align}   
	and
	\begin{align} \label{EqSmallBasin00}
		&\lim_{t\ra \infty } W_2(S[\mu](t),\mu^-) ~=~ 0 \quad \text{if } \int_\R  z\, d\mu(z) < 0 \qaq \\
		&\lim_{t\ra \infty } W_2(S[\mu](t),\mu^+) ~=~ 0 \quad \text{if } \int_\R  z\, d\mu(z)> 0.
	\end{align}   
\end{proposition}
\begin{proof}
	The proof is postponed to Section \ref{SecBasinCompact}.
\end{proof}

The second by-product is the following proposition, which provides useful informations on the 
energy landscape determined by $ \cF $.  
\begin{proposition}\label{PropIntroBasin} 
	Suppose Assumption \ref{AssBasin}. 
	Let $ \mathcal{B}^- $, $ \mathcal{B}^0 $  and $ \mathcal{B}^- $ be the \emph{basins of attraction} of the stationary measures $ \mu^-,\mu^0 $ and $ \mu^+ $, respectively. That is,  
	\begin{align}\label{EqBasinDef}
		\begin{split}
			\mathcal{B}^-&~=~  \{ \mu \in \Pw ~|~  \lim_{t\ra \infty } W_2(S[\mu](t) ,  \mu^-)=0		\},\\
			\mathcal{B}^+&~=~  \{ \mu \in \Pw ~|~  \lim_{t\ra \infty }  W_2(S[\mu](t) ,  \mu^+)=0		\}, \qaq\\
			\mathcal{B}^0&~=~  \{ \mu \in \Pw ~|~  \lim_{t\ra \infty }  W_2(S[\mu](t) ,  \mu^0)=0		\}.
		\end{split}
	\end{align}
	Then, 
	$ \mathcal{B}^- $  and $ \mathcal{B}^+ $ are open subsets of the metric space\footnote{It is shown in  \cite[6.18]{vil} that $ (\Pw, W_2) $ is even a \emph{Polish space}, i.e.\ a complete separable metric space.}
	$ (\Pw, W_2) $, and  $ \mathcal{B}^0 $ is a closed subset of $ (\Pw, W_2) $.
\end{proposition} 
\begin{proof}
	The claim follows from Proposition \ref{PropBasin} and Corollary \ref{CorB0} below.
\end{proof}

The  results of this paper are not completely new. 
Indeed, 
Theorem \ref{ThmConv} and Proposition \ref{PropIntroSmallBasin} have already been obtained in the paper \cite{Tugaut13}\footnote{Note that the interaction term in \cite{Tugaut13} is of  polynomial form, whereas in the present paper we restrict  to the linear interaction from the system \eqref{EqSDEBasin}. Hence, the setting in \cite{Tugaut13} is   more general than here.}.
The proofs in \cite{Tugaut13} are based on methods from the theory of partial differential equations. 
The main contributions of this paper are that we use the Wasserstein framework to prove these results (which provides   shorter proofs than in \cite{Tugaut13}), and that the results hold in the stronger topology of the Wasserstein distance (whereas the results in \cite{Tugaut13} are formulated in terms of the weak topology).
However, to our knowledge, Proposition \ref{PropIntroBasin} is a new result.
It is expected that this proposition will become useful  in the study of the metastable behaviour of the system \eqref{EqIntroMicrSystem} via the Wasserstein framework. The latter is left for  future research.

This paper is organized as follows.  
First, we    
recall some elements of the construction of \emph{Wasserstein gradient flows} in Section \ref{SubsecWassSpace}.
Then,   in Section \ref{SubsecHamil}, we
compare $ \cF $ with the functional  $\bar{H}$, which appeared in \cite{BaMe2019}.
%
In Section \ref{SubsecStat} we characterize the stationary measures, and in Section \ref{SubsecSym} we show a useful symmetry property of the McKean-Vlasov evolution.
In Chapter \ref{SecBasinCompact} we first show  some compactness property of the gradient flows for $ \cF $, and then use this property to  prove Proposition \ref{PropIntroSmallBasin}.
In Chapter \ref{SecBasinOpen} we prove the main part of Proposition \ref{PropIntroBasin}.
In Chapter \ref{SecProofThmConv} we   provide the proof of Theorem \ref{ThmConv} and state some immediate consequences of this theorem for the set $ \mathcal{B}^0 $.
Finally, 
In Section \ref{SecAssu} we briefly discuss  the assumptions we make in this paper.

\section{Preliminaries}\label{SecBasinPrelim}
\subsection{Wasserstein gradient flows}\label{SubsecWassSpace}
In this section, we briefly
recall some elements of the construction of \emph{Wasserstein gradient flows}.
For simplicity, we  restrict all definitions to the functional $ \cF $ from \eqref{EqDefiF}.
For more general functionals and for the details, we refer to \cite{ambgigsav}.

Let $ \cP_2(\R) $ denote the space of all probability measures  on $ \R $, whose second moment is finite. We equip $ \cP_2(\R) $ with the \emph{Wasserstein distance} $ W_2 $, which, for $ \mu, \nu \in \Pw $ is defined by
\begin{align}\label{EqDefiWassDist}
	W_2(\mu,\nu)^2 := \inf_{\gamma \in \CouplW{\mu,\nu} }  \int_{\R^2} |y-y'|^2 \,  d\gamma(y,y') ,
\end{align}
where   $ \CouplW{\mu,\nu} $ denotes the space of all probability measures on $ \R^2 $ that have $ \mu  $ and $ \nu $ as marginals.

Let $(\mu_t)_{t\in[0,\infty)}  $ be a curve of probability measures such that $\mu_t \in \Pw $ for all $t\in[0,\infty)$. Then we say that $(\mu_t)_{t\in[0,\infty)} $ is \emph{absolutely continuous} if  there exists   $ m \in \Lrm^2_{\mathrm{loc}} ((0,\infty)) $ such that 
\begin{align}
	W_2(\mu_s,\mu_t) ~\leq ~
	\int_{s}^{t} m(r) \, dr \qfa 0<s<t<\infty.
\end{align}
We denote   the set of all absolutely continuous curves in $ (\Pw,W_2) $ by $ \ACLio $.
It is shown in \cite[1.1.2]{ambgigsav} that for all $ (\mu_t)_{t\in[0,\infty)} \in \ACLio $,  there exists  $ |\mu'| \in \Lrm^2_{\mathrm{loc}}((0,\infty)) $, called the \emph{metric derivative} of $ (\mu_t)_{t\in[0,\infty)} $, such that 
\begin{align}\label{EqIntroMetricDer}
	|\mu'|(t) ~=~ \lim_{s\ra t} \frac{W_2(\mu_s, \mu_t)}{|s-t|} \ \ \ \ \text{ for almost every } t\in (0,\infty).
\end{align}

Another important object   is the   \emph{metric slope} (cf.\ \cite[1.2.4]{ambgigsav}) of $ \cF $, which  is defined by 
\begin{align}\label{EqIntroMetricSlope}
	|\partial \cF| ( \mu ) ~= ~ \limsup_{\nu \ra \mu}\left( \frac{\cF(\mu ) -\cF(\nu )}{W_2(\mu,\nu)}   \right)^+ \qf \mu \in D(\cF),
\end{align}
and $ |\partial \cF| ( \mu )  = \infty  $ for $ \mu \in \Pw \setminus D(\cF) $.

We are now in the position to define the notion of \emph{Wasserstein gradient flows for} $ \cF $.
There are several different and equivalent ways to do this; some of them are listed in \cite[Chapter 11]{ambgigsav}. 
In this paper, we choose the definition as  a \emph{curve of maximal slope} (cf.\ \cite[1.3.2]{ambgigsav}).
\begin{definition}
	We say that  a curve $ (S[\mu] (t))_{t\in[0,\infty)} \in \ACLio $  is a \emph{(Wasserstein) gradient flow for} $ \cF $ with initial value $ \mu \in \Pw $ if $ \lim_{t\downarrow 0} W_2(S[\mu] (t), \mu)\,=\,0 $, and if the map $ t \mapsto \cF(S[\mu] (t)) $ is locally absolutely continuous in $ (0, \infty) $ with 
	\begin{align}\label{EqMonotonicity}
		\frac{d}{dt} \cF(S[\mu](t)) = -  |\partial \cF|^2(S[\mu](t))  = - \left|(S[\mu])'\right|^2(t) \quad  \text{for almost every }t\in (0,\infty).
	\end{align}
\end{definition}  

We conclude this section with some useful properties of $ \cF $ and Wasserstein gradient flows for $ \cF $,  which we use many times in this paper.
\begin{lemma}\label{LemProperties}
	Suppose Assumption \ref{AssBasin}. Then the following statements are true.
	\begin{enumerate}[(i)] 
		\item \emph{(Lower bound on $ \cF $)}
		There exists $ c>0  $ such that 
		\begin{align}\label{EqBasinLBF}
			\cF(\mu) ~\geq ~ c\left( \int_{\R} |x|^{2 + \epsilon} \, d\mu(x) -1\right)  \qfa \mu \in \Pw.
		\end{align} 
		\item \emph{($ \lambda $-convexity of $ \cF $)} There exists $ \lambda < 0 $ such that $ \cF $ is \emph{$ \lambda $-convex along generalized geodesics} in the sense of \cite[4.0.1]{ambgigsav}.
		\item \emph{(Existence)} For each $ \mu \in \Pw $, there exists a  gradient flow $ (S[\mu] (t))_{t\in[0,\infty)} $ for $ \cF $. 
		\item \emph{(Energy identity)}
		Let $ \mu \in D(\cF) $. Then,  for all $ t\in (0,\infty) $, 
		\begin{align}\label{EqWGFDefi1}
			\begin{split}
				0~=~
				\cF(S[\mu](t)) - 		\cF(\mu) + \frac 12 \int_0^t\big(|\partial \cF|^2(S[\mu](r))  + |(S[\mu])'|^2(r)\big)\,dr .
			\end{split}
		\end{align}
		\item \emph{(Regularization estimate)} 
		Let $ \mu \in \Pw $. Then, 
		\begin{align}\label{EqRegu}
			\cF(S[\mu] (t))   ~\leq~
			\cF(\nu ) 
			+ 
			\frac{\lambda}{2(\eee^{\lambda t} - 1)}\, 
			W_2(\mu,\nu)^2  \qfa \nu \in D(\cF) \text{ and } t\in (0,\infty).
		\end{align} 		
		\item \emph{(contraction and semigroup property)} 
		Let  $ \mu, \nu \in \Pw $.  Then, 
		\begin{align}\label{EqContraction}
			W_2(S[\mu] (t), S[\nu] (t)) ~\leq~ \eee^{-\lambda t} \, W_2(\mu, \nu) \qfa  t\in (0,\infty).
		\end{align}
		In particular, the semigroup property $ S[S[\mu](h)](t) = S[\mu](t+h) $ holds for all $ t,h>0 $. 
	\end{enumerate}
\end{lemma}
\begin{proof}
	Part \emph{(ii)} is proven in \cite[Section 9.3]{ambgigsav} or \cite[3.35]{BaMe2019}, part \emph{(iii)}  in \cite[11.1.3 and 11.2.8]{ambgigsav}, part \emph{(iv)}  in \cite[2.3.3 and 4.0.4]{ambgigsav}, part \emph{(v)}  in \cite[4.3.2]{ambgigsav}\footnote{Note that there is a typo in  \cite[(4.3.2)]{ambgigsav}: It must be $ \frac{\eee^{\lambda T} - 1 }{\lambda} $ instead of $ \frac{\eee^{\lambda T} - 1 }{T} $.}  and part \emph{(vi)} is proven in \cite[(11.2.2)]{ambgigsav}.
	
	It remains to show part \emph{(i)}. 
	Let $ \mu \in D(\cF) $, since otherwise the claim is trivial. In the following let $ C>0 $ denote a constant which does not depend on $ \mu $, and may change from line to line.
	We proceed as in the proof of \cite[3.34]{BaMe2019} and use Assumption \ref{AssBasin} (2) to observe that 
	\begin{align}\label{LBFEqq}
		\begin{split}
			\cF&(\mu) \geq -C  + \frac{1}{4} \int_{  \R^2	}  \Psi(x)  \, d\mu(x) + \frac{1}{2} \int_{  \R^2	} \left(\frac{1}{4} \big(\Psi(x) +  \Psi(\bar{x})\big) - J x \bar{x}			\right) d\mu(x) d\mu(\bar{x})\\ 
			&\geq -C  + \frac{c''}{4} \int_{\R} |x|^{2 + \epsilon} \, d\mu(x)+ \frac{1}{2} \int_{  \R^2	} \left(\frac{c''}{4} \big(|x|^{2+\epsilon} +  |\bar{x}|^{2+\epsilon}\big)- J x \bar{x}			\right) d\mu(x) d\mu(\bar{x}) .
		\end{split}
	\end{align}
	Note that, as a  consequence of the classic Young inequalities, for all $ x, \bar{x} \in \R $ and all $ \alpha>0 $, $ |x \bar{x}| \leq |x|^2 /2 + |\bar{x}|^2/2 $ and $ |x|^{2 + \epsilon} \geq \alpha |x|^{2} -C_\alpha $ for some constant $ C_\alpha >0$ (which only depends  on $ \alpha $ and $ \epsilon $).
	Then, by choosing $ \alpha  $ large enough, we can show that the last term on the right-hand side of  \eqref{LBFEqq} is greater or equal to $ -C_\alpha \frac{c''}{4}$. This concludes the proof. 
\end{proof}

\subsection{Macroscopic Hamiltonians}\label{SubsecHamil}
In this section we first introduce and recall some facts about the   function $ \bar{H}:\R\ra\R $, which was the object of investigation  in the paper \cite{BaMe2019}. 
Then, in Lemma \ref{LemRelationFH}, we show the relation between   $ \cF $ and $ \bar{H} $, and infer from that useful analytic facts about  $ \cF $.

Let the   function $ \varphi^*:\R\ra\R $ be defined by 
\begin{align}\label{EqEpsMuBarStar}
	\varphi^*(\sigma)  ~=~   \log \int_{\R} \eee^{ \sigma z-     \Psi(z) }\,  dz \qf \sigma \in \R.
\end{align}
Let $ \varphi:\R\ra\R $ be the \emph{Legendre transform} of $ \varphi^* $, i.e., 
\begin{align}\label{EqEpsMuBarCramer}
	\varphi(m)~=~ \sup_{\sigma \in \mathbb{R}} \left(\sigma m - \varphi^*(\sigma) \right) \qf m \in \R.
\end{align}
It is then  well-known from standard properties of Legendre transforms (see for instance \cite[III.2.5]{PatThesis}, \cite[A.1.1]{BaMe2019} or \cite[Lemma 41]{GORV}) that for all $ m, \sigma  \in \R $, 
\begin{align}\label{EqLegCramer}
	\varphi'(m)m-\varphi^*(\varphi'(m)) = \varphi(m), \qquad  (\varphi^*)'(\varphi'(m)) = m \qaq (\varphi^*)'(\sigma)= \int_{\R} z \, d\mu^{\sigma}, 
\end{align}
where, 
for $ \sigma \in \R $,   the probability measure $ \mu^{ \sigma} \in \Pw $ is defined  by 
\begin{align} \label{EqEpsGCMeasure}
	d\mu^{ \sigma}(z)  ~=~ \eee^{-\varphi^*(\sigma) +\sigma z-  \Psi(z)} \, dz~=~ 	\frac{ \eee^{ \sigma z-  \Psi(z)} }{ \int_{\R} \   \eee^{\sigma \bar{z}- \Psi(\bar{z})} \, d\bar{z}}\, dz.
\end{align} 
Finally,  we define the function $ \bar{H}:\R\ra\R $ by
\begin{align}\label{EqEpsHBar}
	\bar{H}(z) ~=~ \varphi(z) ~-~   \frac{J}{2}z^2 \qf z\in \R.
\end{align}
\begin{remark}\label{RemarkBarH}
	The function $ \bar{H} $ played the role of the \emph{macroscopic Hamiltonian} in \cite{BaMe2019}, where  the metastable behaviour of the system \eqref{EqIntroMicrSystem} was studied.
	It is important to notice that in \cite{BaMe2019} the \emph{empirical mean} was chosen to be the \emph{macroscopic order parameter}.
	Recall from fact (i) and (ii) of the introduction that the functional $ \cF $ appears as the macroscopic Hamiltonian of the system \eqref{EqIntroMicrSystem} by choosing the \emph{empirical distribution} as the macroscopic order parameter; see   \cite[Section 1.4]{BaMe2019} for more details on this. 
	
	Moreover, as it is shown in   \cite[3.4]{BaMe2019}, under Assumption \ref{AssBasin}, the function $\bar{H} $ admits exactly three critical points, which are located at~$- m^\star  $,~$0$ and~$ m^\star  $ for some  $ m^\star>  0$.
	Furthermore,   $ \bar{H}''(0) < 0 $, $ \bar{H}''(m^\star  ) = \bar{H}''(- m^\star  ) >0 $, and $ \bar{H}(0) > \bar{H} (m^\star  ) = \bar{H} (- m^\star  ) $.
	That is,  $\bar{H} $ has a local maximum at~$0$, and   the two global minima of $\bar{H} $ are located  at~$\pm m^\star  $.
\end{remark}

In the following let $ \mr[\mu] = \int_\R zd\mu(z)$ denote the mean of a probability measure $ \mu \in \Pw $.
%
We have the following relation\footnote{See also \cite[Section IV.2]{PatThesis} for a  more general result.} between the macroscopic Hamiltonians $ \cF $ and $\bar{H}$.
\begin{lemma} \label{LemRelationFH}
	Suppose Assumption \ref{AssBasin}.
	Then, for all $ m \in \R $, we have that 
	\begin{align}\label{EqRelationFH1}
		\cF(\mu)
		~>~ \cF(\mu^{\varphi'(m)}) \qfa \mu \in 
		\Pw  \text{ such that } \mr[\mu] = m \text{ and } \mu \neq \mu^{\varphi'(m)},
	\end{align}
	and, 
	\begin{align}\label{EqRelationFH}
		\bar{H} (m) ~=~
		\min_{\mu \in 
			\Pw , \mr[\mu] = m} \cF(\mu)
		~=~ \cF(\mu^{\varphi'(m)}).
	\end{align} 
	Moreover, let 
	\begin{align}\label{EqBasinStationaryMeasures}
		\mu^- := \mu^{\varphi'(-m^\star)} , \qquad  \mu^0 := \mu^{\varphi'(0)} \qaq 
		\mu^+ := \mu^{\varphi'(m^\star)} .
	\end{align} 
	Then, $ \cF $ admits exactly two global minima, one at $ \mu^-    $ and one at  $ \mu^+    $, and we have that 
	$ \cF(\mu^- ) =  \cF(\mu^+ ) < \cF(\mu^0 )  $.
\end{lemma}
\begin{proof} 
	If $ \cF(\mu) = \infty $, then \eqref{EqRelationFH1} is trivially satisfied. So we assume that $ \cF(\mu) < \infty $.
	In the following let $ \cH(\cdot|\cdot) $ denote  the \emph{relative entropy functional} (see e.g.\ \cite[9.4.1]{ambgigsav}), and let $ \mr[\mu] = m $.
	Then, by using \eqref{EqLegCramer} and by denoting the Lebesgue density of $ \mu $  by $ \rho $,
	\begin{align}\label{EqRelationFH11}
		\begin{split}
			\cF(\mu ) &~=~
			\int_\R
			\log(\rho \, \eee^\Psi) d\mu
			- \frac{J}{2} m^2
			~=~ \cH(\mu \, | \, \mu^{\varphi'(m)}) +
			\varphi'(m)m-\varphi^*(\varphi'(m)) - \frac{J}{2} m^2 \\
			&~=~\cH(\mu \, | \, \mu^{\varphi'(m)}) +  \bar{H} (m) . 
		\end{split}
	\end{align}
	Since $ \cH(\mu^{\varphi'(m)} \, | \, \mu^{\varphi'(m)}) = 0 $ and $ \cH(\mu \, | \, \mu^{\varphi'(m)}) > 0 $ if $  \mu \neq \mu^{\varphi'(m)} $,
	\eqref{EqRelationFH11} implies that 
	\begin{align}\label{EqRelationFH111}
		\begin{split}
			\cF(\mu )  > \bar{H} (m) \text{ if } \mu \neq \mu^{\varphi'(m)} \qaq \cF(\mu^{\varphi'(m)} )  ~=~ \bar{H} (m).  
		\end{split}
	\end{align} 
	From \eqref{EqRelationFH111} we immediately infer \eqref{EqRelationFH1} and \eqref{EqRelationFH}. 
	Finally, \eqref{EqRelationFH} and  Remark \ref{RemarkBarH} imply the last two claims.
\end{proof}

\subsection{Stationary points of the McKean-Vlasov evolution}\label{SubsecStat}

In this section we characterize the \emph{stationary points} of the McKean-Vlasov evolution\footnote{See also \cite{HerrTug} for similar results.}, where we say that $ \mu \in \Pw $ is \emph{stationary} if 
\begin{align}\label{EqStationary}
	S[ \mu](t) =   \mu  	 \qfa t \in (0,\infty),
\end{align}
or equivalently, 
\begin{align}\label{EqStationary0}
	|(S[\mu])'| (t) ~=~0 \qf \text{almost every }   t\in (0,\infty).
\end{align}

\begin{lemma}\label{LemStationary}
	Suppose Assumption \ref{AssBasin}.
	Let $ \mu \in \Pw $.
	Then, the following statements are equivalent. 
	\begin{enumerate}[(i)]
		\setlength\itemsep{-0.1em} 
		\item $ \mu $ is stationary. 
		\item $ |\partial \cF|(\mu) = 0$. 
		\item  $ \mu \in \{\mu^-,\mu^0, \mu^+\} $. 
	\end{enumerate}
\end{lemma}
\begin{proof}
	$ (i) \Rightarrow (ii) $.
	Suppose that $ \mu  $ is stationary.  
	Then, combining \eqref{EqMonotonicity} (which  holds  true even if $ \mu \notin D(\cF) $) and \eqref{EqStationary0}, we infer that  $  |\partial \cF| (S[\mu](t)) = |(S[\mu])'| (t) =0$ for almost every $  t\in (0,\infty)  $.
	Then, by the lower semi-continuity of $ |\partial \cF| $ (see \cite[2.4.10]{ambgigsav}]) and the fact that $ \lim_{t\downarrow 0} W_2(\mu,S[\mu](t) )=0 $, we conclude that $ |\partial \cF|(\mu) ~\leq~ \liminf_{t\downarrow 0 } |\partial \cF| (S[\mu](t)) ~=~0 . $ 
	%
	
	$ (ii) \Rightarrow (iii) $.
	Using \cite[10.4.13]{ambgigsav}, we have  that the Lebesgue density $ \rho  $ of $ \mu  $ belongs to the Sobolev space $ W^{1,1}_{loc}(\R) $\footnote{More precisely, in \cite[10.4.13]{ambgigsav} it is shown that $ L_F(\rho) \in W^{1,1}_{loc}(\R) $ for some function $ L_F: [0,\infty) \ra [0,\infty) $, which is defined right after \cite[(10.4.17)]{ambgigsav}. 
		However, in our case we have that $ L_F(z) =z  $ for all $ z\in [0,\infty) $.}. 
	Let $ m = \mr[\mu] $. 
	Then, by using again \cite[10.4.13]{ambgigsav},
	\begin{align}\label{EqStationary1}
		|\partial \cF|^2(\mu)   =  
		\int_{\R} \left| \frac{\partial_z \rho(z)}{\rho(z)} + \Psi'(z) - J m\right| ^2 d \mu(z)  = \int_{\R} \left| \frac{\partial_z \left( \rho(z) \eee^{\Psi(z) - Jm z}\right) }  {\rho(z) \eee^{\Psi(z) - Jm z} }   \right| ^2 d \mu(z). 
	\end{align}
	Since  $ |\partial \cF|(\mu) = 0$,   \eqref{EqStationary1} implies that  the map $ z \mapsto  \rho(z) \eee^{\Psi(z) - Jm z} $ is  constant $ \mu-$almost everywhere.
	Therefore,   for $ \mu-$a.e.\ $z, z' \in \R $,
	\begin{align}\label{EqStationary20}
		\rho(z) ~=~ \rho(z') \,\eee^{\Psi(z') -Jmz'}\,  \eee^{-\Psi(z)+Jm z }  .
	\end{align}
	By fixing $z' \in \R $ and by using  the definition of $ \varphi^* $ and  that $ \int_{\R} \rho(z) \, dz =1 $, \eqref{EqStationary20} implies that 
	\begin{align}\label{EqStationary2}
		\rho(z) 
		~=~ \eee^{- \varphi^*(Jm) }\, \eee^{-\Psi(z)+Jm z }.
	\end{align}
	In particular, combining \eqref{EqLegCramer} and \eqref{EqStationary2} yields that $ m =  (\varphi^*)'(Jm) $.
	And by using the second claim in \eqref{EqLegCramer}, we infer that $ \bar{H}'(m)= 0  $. 
	However, in Remark \ref{RemarkBarH} we have seen that there are only three solutions to this equation. This implies that 
	\begin{align}\label{EqStationary3}
		m ~\in ~ \{-m^\star, 0, m^\star\}.
	\end{align}
	Combining \eqref{EqStationary2} and \eqref{EqStationary3} yields part \emph{(iii)}.

	$ (iii) \Rightarrow (ii) $. 
	Combining the representation \eqref{EqStationary1} with the definition of the measures $ \mu^-,\mu^0 $ and $ \mu^+ $ yields part \emph{(ii)}.

	$ (ii) \Rightarrow (i) $. 
	From \cite[2.4.15 and 2.4.16]{ambgigsav}, we have that for all $ t >0 $,  
	\begin{align}\label{EqStationary4}
		|\partial \cF| (S[ \mu](t)) ~\leq~ \eee^{-\lambda t} |\partial \cF|(\mu)   ~=~  0 , 
	\end{align}
	where the parameter  $ \lambda  $ was introduced in Lemma \ref{LemProperties}.
	Again, using that 
	$  |(S[\mu])'| (t) = |\partial \cF| (S[\mu](t)) $ for almost every $  t\in (0,\infty)  $, \eqref{EqStationary4}  yields part \emph{(i)}.
\end{proof}

\subsection{Symmetry property}\label{SubsecSym}

In this section we show that  gradient flows for $ \cF $ admit a useful symmetry property. In the following we denote by  $ f_\#\mu  $  the image measure of a   measure $ \mu $ under a Borel map $ f $. 
\begin{lemma}\label{LemBSymmetry}
	Let $ \varsigma:\R\ra\R $ be defined by $ \varsigma(z)=-z $, and let $ \mu \in \Pw $.
	Then,    
	\begin{align}\label{EqBSymmetry}
		S[\varsigma_\#\mu](t) =  \varsigma_\#S[\mu](t) 	 \qfa t \in (0,\infty). 
	\end{align}
\end{lemma}
\begin{proof}
	First note that 
	\begin{align}\label{EqBSymmetry0}
		\cF(\nu ) = \cF(\varsigma_\#\nu ) \qfa  \nu \in \Pw ,
	\end{align}
	and therefore,
	\begin{align}\label{EqBSymmetry1}
		|\partial \cF|(\nu ) = |\partial \cF|(\varsigma_\#\nu )\qfa  \nu \in \Pw.
	\end{align}
	Moreover, for all $ \nu \in \ACLio $ and $ 0 < s < t < \infty  $,
	\begin{align}\label{EqBSymmetry11}
		W_2(\nu_s,\nu_t)
		~= ~	W_2(\varsigma_\#(\varsigma_\#\nu_s),\varsigma_\#(\varsigma_\#\nu_t)) 
		~\leq ~	W_2(\varsigma_\#\nu_s,\varsigma_\#\nu_t)
		~\leq ~
		W_2(\nu_s,\nu_t).
	\end{align}
	Therefore, $ W_2(\nu_s,\nu_t)	= 	W_2(\varsigma_\#\nu_s,\varsigma_\#\nu_t)$, and we have that the metric derivatives coincide, i.e.,  
	\begin{align}\label{EqBSymmetry2}
		|\nu'| (t)= |(\varsigma_\#\nu)'|(t) \,  \text{ for almost every }   t\in (0,\infty) \text{ and  for all $ \nu \in \ACLio $} .
	\end{align}
	
	Then, by combining \eqref{EqMonotonicity} (which holds true even if $ \mu \notin D(\cF) $), \eqref{EqBSymmetry0}, \eqref{EqBSymmetry1} and \eqref{EqBSymmetry2}, we have that for almost every $ t\in (0, \infty) $,
	\begin{align}\label{EqBSymmetry22}
		\begin{split}
			\frac{d}{dt} \cF(\varsigma_\#S[\mu](t)) &~= ~ \frac{d}{dt} \cF(S[\mu](t)) = - \left|(S[\mu])'\right|^2(t)= - \left|(\varsigma_\#S[\mu])'\right|^2(t) , \qaq \\ 
			\frac{d}{dt} \cF(\varsigma_\#S[\mu](t)) &~= ~ \frac{d}{dt} \cF(S[\mu](t)) = - |\partial \cF|^2(S[\mu](t)) = -|\partial \cF|^2(\varsigma_\#S[\mu](t)) .
		\end{split}
	\end{align}
	Moreover, by using the same arguments as in \eqref{EqBSymmetry11}, we infer that
	\begin{align}\label{EqBSymmetry21}
		\lim_{t\downarrow 0} W_2(\varsigma_\#S[\mu](t) ,\varsigma_\#\mu  ) ~=~\lim_{t\downarrow 0} W_2(S[\mu](t) ,\mu  )~=~0.
	\end{align}
	Combining \eqref{EqBSymmetry22} and \eqref{EqBSymmetry21} yields that the curve $ (\varsigma_\#S[\mu](t))_{t\in(0, \infty) } $ is the gradient flow for $ \cF $ with initial value $ \varsigma_\#\mu $. 
\end{proof} 
\section{Convergence in the valleys}\label{SecBasinCompact}
In this chapter we first show some compactness property of the McKean-Vlasov paths in Lemma \ref{LemCompactness}.
Then, we use this result to prove Proposition \ref{PropIntroSmallBasin}.
\begin{lemma}\label{LemCompactness}
	Suppose Assumption \ref{AssBasin}.
	Let $ \mu \in D(\cF) $. 
	Then, there exist a  sequence $ (t_k)_k $ and $ \mu^* \in \{\mu^-,\mu^0, \mu^+\} $ such that $ \lim_{k\ra\infty  } t_k = \infty $, 
	\begin{align}
		\lim_{k\ra \infty } W_2(S[\mu](t_k),\mu^*) ~=~
		0 \qaq \lim_{t\ra \infty } \cF(S[\mu](t)) ~=~ \cF(\mu^* ).
	\end{align} 
\end{lemma}
\begin{proof}
	In the following let $ \mu_t = S[\mu](t)  $. 
	We prove this lemma in three steps.

	\textbf{Step 1.} [There exists a subsequence $ (t_n)_n $ such that $ \lim_{n\ra \infty } |\partial \cF| ( \mu_{t_n}) =0 $.]
	
	\noindent
	Note that the sequence $ (\cF  ( \mu_{t}))_{t\in [0,\infty)} $ is a continuous, monotone and bounded sequence of real numbers by    \eqref{EqMonotonicity} and \eqref{EqBasinLBF}. 
	Therefore, it converges, as $ t \ra \infty $, to a number $ L^* \in \R $. 
	In particular, by \eqref{EqMonotonicity}, 
	\begin{align}
		\begin{split}
			\int_0^\infty |\partial \cF|^2 ( \mu_{r})\, dr
			&~=~-
			\int_0^\infty  \frac{d}{dr} \cF ( \mu_{r})\, dr
			~=~
			-L^* + \cF ( \mu) 
			~<~\infty.
		\end{split}
	\end{align}
	This implies the claim of Step 1. 
	
	\textbf{Step 2.} [$ \lim_{k\ra \infty } W_2(\mu_{t_{n_k}},\mu^*) $ for some $ \mu^* \in \{\mu^-,\mu^0, \mu^+\} $ and a  subsubsequence $ (t_{n_k})_k $.]
	
	\noindent 
	By \eqref{EqBasinLBF},  the monotonicity of $ t \mapsto \cF(\mu_t) $  and the fact that $ \mu_0 = \mu \in D(\cF) $, we have that 
	\begin{align}\label{EqComp1}
		\sup_{n\in \N} \int_{\R} |x|^{2 + \epsilon} \, d\mu_{t_n}(x) ~\leq~  \sup_{n\in \N}\left( \frac{1}{c} \cF(\mu_{t_n}) + 1\right) 
		~\leq~    \frac{1}{c} \cF(\mu ) + 1 ~<~ \infty. 
	\end{align}
	Using \cite[6.8 (iii)]{vil}, this implies that there exist  a further subsequence $ (t_{n_k})_k $ and $ \mu^* \in \Pw $ such that $ \lim_{k\ra \infty } W_2(\mu_{t_{n_k}},\mu^*) $.
	It remains to show that   $ \mu^* \in \{\mu^-,\mu^0, \mu^+\} $.
	In order to do this, we use the lower semi-continuity of $ |\partial \cF| $ (\cite[2.4.10]{ambgigsav}]) and Step 1 to observe that  
	\begin{align}
		|\partial \cF|(\mu^*) ~\leq~ \liminf_{k\ra \infty } |\partial \cF| ( \mu_{t_{n_k}}) ~=~0 .
	\end{align}
	Combining this with Lemma \ref{LemStationary} yields the claim of Step 2.
	
	\textbf{Step 3.} [$\lim_{t\ra \infty } \cF(\mu_t) = \cF(\mu^* ) $.]
	
	\noindent
	First note that by the lower semi-continuity of $ \cF $ (see \cite[3.35]{BBGradFlow} or \cite[Section 9.3]{ambgigsav}]), we have that 
	\begin{align}
		L^* = \lim_{t\ra \infty } \cF(\mu_t)= \lim_{k\ra \infty } \cF(\mu_{t_{n_k}}) \geq    \cF(\mu^* ).
	\end{align}
	To show the other inequality, we use \cite[2.4.9]{ambgigsav}, and observe that for all 
	$ k \in \N $,
	\begin{align}\label{EqComp}
		|\partial \cF| ( \mu_{t_{n_k}}) ~\geq ~ \left( \frac{\cF(\mu_{t_{n_k}}) -\cF(\mu^* )}{W_2(\mu_{t_{n_k}},\mu^*)} + \frac \lambda 2 W_2(\mu_{t_{n_k}},\mu^*) \right)^+ , 
	\end{align}
	where the parameter  $ \lambda  $ was introduced in Lemma \ref{LemProperties}.
	Note that \eqref{EqComp} is equivalent to 
	\begin{align}
		W_2(\mu_{t_{n_k}},\mu^*)\, |\partial \cF| ( \mu_{t_{n_k}}) ~\geq ~ \left(  \cF(\mu_{t_{n_k}}) -\cF(\mu^* )  + \frac \lambda 2 W_2^2(\mu_{t_{n_k}},\mu^*) \right)^+ .
	\end{align}
	Taking the limit as $ k\ra \infty $ on both sides, and using Step 1 and Step 2,  implies that 
	\begin{align}
		0 ~\geq ~ \left(  L^*-\cF(\mu^* )   \right)^+ .
	\end{align}
	We conclude that $ L^*  \leq    \cF(\mu^* ). $
\end{proof}

With this compactness result in hand, we are able to prove Proposition \ref{PropIntroSmallBasin}.

\begin{proofofp}\emph{\ref{PropIntroSmallBasin}.}
	In the following let $ \mu_t = S[\mu](t)  $.
	It suffices to consider only the case that $ \mr[\mu] < 0 $.
	We know from Lemma \ref{LemCompactness} that there exists a subsequence  $ (\mu_{t_k})_k $     such that 
	\begin{align}
		\lim_{k\ra\infty  }W_2(\mu_{t_k},\mu^*) = 
		0  \qaq \lim_{t\ra \infty } \cF(\mu_{t}) = \cF(\mu^* ) \quad \text{ for some }\mu^* \in \{\mu^-,\mu^0, \mu^+\}.
	\end{align}
	We first show that $\mu^*=\mu^- $ (which implies   \eqref{EqSmallBasin0}), and then show that $\lim_{t\ra\infty  }W_2(\mu_{t },\mu^-) =    0    $ (which implies   \eqref{EqSmallBasin00}).
	
	\textbf{Step 1.} [ $\mu^*=\mu^- $. ]
	
	\noindent
	We  show that the cases $\mu^*=\mu^+ $ or $\mu^*=\mu^0 $ lead to contradictions.
	First suppose that $\mu^*=\mu^+ $. 
	Since the map $ t \mapsto \mr[\mu_t] $ is continuous and since $ \mr[\mu_0]=\mr[\mu]<0 $, we have that there exists $ t' \in (0,\infty) $ such that $ \mr[\mu_{t'}]=0 $.
	Then, by the monotonicity of $ t \mapsto \cF(\mu_t) $ and by Lemma \ref{LemRelationFH},  
	\begin{align}\label{EqSmallBasin}
		\cF(\mu^0)~\geq~ \cF(\mu) ~\geq~ \cF(\mu_{t'}) ~\geq~ \cF(\mu^0).
	\end{align}
	Hence,  $ \cF(\mu_{t'})= \cF(\mu) =\cF(\mu^0)$.
	Combining this with \eqref{EqWGFDefi1}, implies that $ \int_0^{t'}|\partial \cF|^2(\mu_r)  \,dr =0 $. This in turn yields that $ |\partial \cF|(\mu_r)  =0 $ for almost every $ r \in (0,t') $. Then, by the lower semi-continuity of $ |\partial \cF| $ (\cite[2.4.10]{ambgigsav}]),
	we infer that $ |\partial \cF|(\mu) ~\leq~ \liminf_{r\downarrow 0 } |\partial \cF| ( \mu_{r}) ~=~0 . $
	Therefore, we have that 
	\begin{align}\label{EqSmallBasin11}
		|\partial \cF|(\mu)  ~=~0 \qaq  \cF(\mu) ~=~\cF(\mu^0). 
	\end{align}
	By Lemma \ref{LemRelationFH} and Lemma \ref{LemStationary}, \eqref{EqSmallBasin11} implies   that $ \mu = \mu^0 $. 
	This yields to a contradiction, since  $ \mr[\mu]<0 $.
	The case $\mu^*=\mu^0 $ is treated analogously. 
	
	\textbf{Step 2.} [ $\lim_{t\ra\infty  }W_2(\mu_{t },\mu^-) =    0    $. ]
	
	\noindent
	Let $ (\mu_{s_n})_{n\in \N} $ be any subsequence of $ (\mu_{t})_{t\in[0,\infty)} $.
	Using the same compactness argument from Step 2 of the proof of Lemma \ref{LemCompactness}, we know that there exists a further subsequence $ (\mu_{s_{n_k}})_{k\in \N} $ such that $ \lim_{k\ra\infty  }W_2(\mu_{s_{n_k}},\mu') = 
	0  $ for some $ \mu' \in \Pw $. 
	In order to show the claim of Step 2, it remains to show that $ \mu' = \mu^- $.    
	First we notice that
	\begin{align}\label{EqSmallBasin1}
		\cF(\mu') ~\leq~ \liminf_{k\ra\infty} \cF(\mu_{s_{n_k}})
		~=~ \lim_{t\ra \infty } \cF(\mu_{t}) ~=~ \cF(\mu^- ). 
	\end{align}
	In view of Lemma \ref{LemRelationFH}, this implies that either $ \mu'= \mu^- $ or $ \mu'=\mu^+ $. 
	We now use similar arguments as in Step 1 to show that the latter case yields to a contradiction. 
	So suppose that $ \mu'=\mu^+ $. 
	Then, since $ \lim_{k\ra\infty  }W_2(\mu_{s_{n_k}},\mu^+) = 
	0  $ and $ \mr[\mu^+] > 0 $, there exists $ s' \in (0, \infty) $ such that $ \mr[\mu_{s'}] > 0 $.
	By the continuity of the map $ t \mapsto \mr[\mu_t] $ and since $ \mr[\mu_0]=\mr[\mu]<0 $, there must be a $ t' \in (0,s') $ such that $ \mr[\mu_{t'}]=0 $.
	Now we use the same arguments as in Step 1 to conclude  \eqref{EqSmallBasin11}.
	This in turn implies   that $ \mu = \mu^0 $, which yields to a contradiction, since  $ \mr[\mu]<0 $.
	This concludes  the proof.
\end{proofofp}

\section{Basin of attraction}\label{SecBasinOpen}

\begin{proposition}\label{PropBasin} 
	Suppose Assumption \ref{AssBasin}.
	Recall the definition of $ \mathcal{B}^- $  and $ \mathcal{B}^+ $ from \eqref{EqBasinDef}.
	Then, 
	$ \mathcal{B}^- $  and $ \mathcal{B}^+ $ are open subsets of $ \Pw $.
\end{proposition} 
\begin{proof}
	In view of Lemma \ref{LemBSymmetry}, it suffices to show the claim only for $ \mathcal{B}^- $.
	In the following abbreviate $ \Delta := \cF(\mu^0 ) - \cF(\mu^- ) $ and recall the definition of $ \lambda <0 $ from Lemma \ref{LemProperties}.
	
	Let $ \nu \in \mathcal{B}^- $.
	That is, $ \nu \in \Pw $ (in particular, it may be that $ \nu \notin D(\cF) $) and we have that $ \lim_{t\ra \infty} W_2(S[\nu](t),\mu^-) =0$. 
	Let $ h \in (0,\infty) $. 
	Note that, by using \eqref{EqRegu} and the semigroup property   (see  Lemma \ref{LemProperties}), we have that 
	\begin{align}\label{EqBasin} 
		S[\nu](h) \in D(\cF) \qaq \lim_{t\ra \infty} W_2(S[S[\nu](h)](t),\mu^-) = \lim_{t\ra \infty} W_2(S[\nu](t+h),\mu^-) =0, 
	\end{align}
	Then,  applying  Lemma \ref{LemCompactness} with $ \mu =  S[\nu](h)$ implies that  
	\begin{align}\label{EqBasin1} 
		\cF(\mu^- ) ~=~ \lim_{t\ra\infty}\cF(S[S[\nu](h)](t))~=~ \lim_{t\ra\infty}\cF(S[\nu](t+h))  ~=~ \lim_{t\ra\infty}\cF(S[\nu](t))  .
	\end{align}
	Therefore, we have that 
	\begin{align}\label{EqBasin11} 
		\lim_{t\ra \infty} W_2(S[\nu](t),\mu^-) =0 \qaq \lim_{t\ra\infty}\cF(S[\nu](t)) ~=~\cF(\mu^- )  .
	\end{align}
	This implies that there exists some $ t'>0 $ such that for all $ t \geq t' $,
	\begin{itemize}
		\setlength\itemsep{-0.2em} 
		\item $ W_2(S[\nu](t),\mu^-) ~\leq~ \frac{1}{4} m^\star $,
		\item $ \cF(S[\nu](t)) ~\leq~ \cF(\mu^- ) + \frac{1}{4}  \Delta $, and 
		\item  $ \eee^{\lambda t}~=~\eee^{-|\lambda| t} ~\leq~ \frac 1 2 $.
	\end{itemize}
	Set
	\begin{align}
		\delta ~:=~ \min\left\lbrace~ \eee^{2\lambda t'}\,  \frac{m^\star}{4} ~,~ \sqrt{\eee^{2\lambda t'} \, \frac{1}{|\lambda|}\, \frac{\Delta}{4}  }  ~\right\rbrace . 
	\end{align}
	We now show that $ B_\delta(\nu) =\{\mu \in \Pw \,|\, W_2(\mu,\nu)< \delta \} \subset  \mathcal{B}^- $.
	Let $ \mu \in B_\delta(\nu) $.
	We   have to show that $ \lim_{t\ra \infty } S[\mu](t) =  \mu^- $.
	In view of Proposition \ref{PropIntroSmallBasin}, it suffices to show that 
	\begin{enumerate}[(i)]
		\setlength\itemsep{-0.2em} 
		\item
		$ \mr[S[\mu](2t')] < 0 $, and that
		\item $ \cF(S[\mu](2t')) ~\leq~ \cF(\mu^0) $.
	\end{enumerate}
	In order to show (i), note that by the contraction estimate \eqref{EqContraction} and the definition of $ t'  $ and $ \delta $, 
	\begin{align}
		W_2(S[\mu](2t'), \mu^-)
		~\leq~
		W_2(S[\nu](2t'), \mu^-) + \eee^{-2\lambda t'} \delta 
		~\leq~ \frac{m^\star}{2}.
	\end{align}
	This implies claim (i).
	To show claim (ii), we use the regularization estimate  \eqref{EqRegu}, 
	and obtain that 
	\begin{align}
		\begin{split}
			\cF(S[\mu](2t')) &~\leq~
			\cF(S[\nu](t')) 
			+ 
			|\lambda|\, 
			W_2(S[\nu]( t'),S[\mu]( t'))^2~\leq~
			\cF(\mu^- ) + \frac{1}{2}  \Delta  ~<~   \cF(\mu^0 ).
		\end{split}
	\end{align}
	This concludes the proof of claim (ii).
\end{proof}

\section{Proof of Theorem \ref{ThmConv}}\label{SecProofThmConv}

\begin{proofoft}\emph{\ref{ThmConv}.}
	In the following let $ \mu_t = S[\mu](t)  $ for all $ t\in [0,\infty) $. 
	First suppose that $ \mu \in   D(\cF) $.
	We know from Lemma \ref{LemCompactness} that there exists a subsequence  $ (\mu_{t_k})_k $ and $ \mu^* \in \{\mu^-,\mu^0, \mu^+\} $ such that 
	\begin{align}\label{EqThmConv}
		\lim_{k\ra\infty  }W_2(\mu_{t_k},\mu^*) = 
		0  \qaq \lim_{t\ra \infty } \cF(\mu_{t}) = \cF(\mu^* ).  
	\end{align}
	
	Let $ (\mu_{s_n})_{n\in \N} $ be a  subsequence of $ (\mu_{t})_{t\in[0,\infty)} $.
	As in Step 2 of the proof of Lemma \ref{LemCompactness}, we infer the existence of a further subsequence, still denoted by  $ (\mu_{s_n})_{n\in \N} $, such that 
	\begin{align}\label{EqThmConv1}
		\lim_{n\ra\infty  }W_2(\mu_{s_{n}},\nu^*) = 
		0    \quad \text{ for some }  \nu^* \in \Pw  . 
	\end{align}
	It remains to show that $ \nu^* = \mu^* $.
	We divide the proof into the three cases $ \mu^* = \mu^- $, $ \mu^* = \mu^0 $ and $ \mu^* = \mu^+ $.

	\textbf{Case 1.} [ $\mu^* = \mu^- $. ]
	
	\noindent
	As in \eqref{EqSmallBasin1}, we infer that $ \cF(\nu^*)\leq \cF(\mu^-) $. 
	By Lemma \ref{LemRelationFH}, this implies that either  $ \nu^* = \mu^- = \mu^* $ or $ \nu^* = \mu^+ $.
	It remains to show that the latter case leads to a contradiction. 
	Note that  by \eqref{EqThmConv} and \eqref{EqThmConv1},
	\begin{itemize}
		\setlength\itemsep{-0.2em} 
		\item there exists $   T>0 $ such that $  \cF(\mu_t) \in [\cF(\mu^-), \cF(\mu^0)) $ for all $ t\geq T $, 
		\item there exists $   N\in\N $ such that $  s_{n } \geq T $ and $ \mr[\mu_{s_{n }}]>0 $ for all $ n\geq N $, and 
		\item there exists $   K\in\N $ such that $  t_k > s_N $ and $ \mr[\mu_{t_k}]<0 $ for all $ k\geq K $.
	\end{itemize}
	In particular, we have that 
	\begin{align}
		\cF(\mu_t) < \cF(\mu^0)    \text{ for all } t\in [s_N,t_K], \qquad  
		\mr[\mu_{s_{N }}]>0  , \qaq 
		\mr[\mu_{t_K}]<0  .
	\end{align} 
	Hence, there exists $ t'\in [s_N,t_K] $ such that  $  \cF(\mu_{t'}) < \cF(\mu^0) $  and $ \mr[\mu_{t'}]=0 $.
	This contradicts Lemma \ref{LemRelationFH}.
	
	\textbf{Case 2.} [ $\mu^* = \mu^+ $. ]
	
	\noindent
	This case is treated in the same way as Case 1.

	\textbf{Case 3.} [ $\mu^* = \mu^0 $. ]
	
	\noindent
	In this case we have that $ \cF(\nu^*)\leq \cF(\mu^0) $.
	There are three subcases given by  $ \mr[\nu^*]=0 $, $ \mr[\nu^*]>0 $ and $ \mr[\nu^*]<0 $.
	
	\textbf{Case 3.1.} [ $\mr[\nu^*]=0$. ]
	
	\noindent
	By Lemma \ref{LemRelationFH}, the combination of $ \cF(\nu^*)\leq \cF(\mu^0) $ and $ \mr[\nu^*]=0 $ yields that $  \nu^*=\mu^0=\mu^*  $.
	
	\textbf{Case 3.2.} [ $\mr[\nu^*]<0$. ]
	
	\noindent
	From Proposition \ref{PropIntroSmallBasin} we know that $ \nu^* \in \mathcal{B}^- $.
	Hence, by Proposition \ref{PropBasin}, there exists $ \delta >0 $ such that 
	$ B_\delta(\nu^*) \subset \mathcal{B}^- $. 
	In particular, by  \eqref{EqThmConv1}, there exists $ N \in \N $ such that $ \mu_{s_N} \in \mathcal{B}^- $.
	This contradicts \eqref{EqThmConv}. Indeed,   the fact that 
	$ \mu_{s_N} \in \mathcal{B}^- $ implies that 
	\begin{align}
		\lim_{t \ra \infty } \mu_{s_N+t}~=~\lim_{t \ra  \infty } S[\mu_{s_N}](t) ~=~ \mu^- \quad \text{ in } \Pw,
	\end{align}
	which contradicts the fact that  
	$ \lim_{k \ra  \infty } \mu_{ t_k}~=~  \mu^* ~=~ \mu^0 \quad \text{ in } \Pw.  $
	
	\textbf{Case 3.3.} [ $\mr[\nu^*]>0$. ]
	
	\noindent
	This case is treated in the same way as Case 3.2. This concludes the proof of this theorem  for the case $ \mu \in  D(\cF) $.
	
	Now let $ \mu \in \Pw \setminus D(\cF) $. Let $ h \in (0,\infty) $. 
	Applying the regularization estimate \eqref{EqRegu} yields  that $ S[\mu](h) \in D(\cF) $. 
	Hence, we have proven the claims  of this theorem for $ S[\mu](h)   $.
	That is, there exists $ \mu^* \in \{\mu^-,\mu^0, \mu^+\} $ such that  
	$ \lim_{t\ra\infty  }W_2(S\big[S[\mu](h)\big] (t) ,\mu^*) = 
	0  $ and $ \lim_{t\ra \infty } \cF(S\big[S[\mu](h)\big] (t)) = \cF(\mu^* ).  $ 
	Therefore, since 
	$ \lim_{t\ra\infty  }W_2(S\big[S[\mu](h)\big] (t) ,\mu^*) = \lim_{t\ra\infty  }W_2(S[\mu] (t) ,\mu^*)  $ and $ \lim_{t\ra \infty } \cF(S\big[S[\mu](h)\big] (t)) = \lim_{t\ra \infty } \cF(S[\mu](t)) $, we conclude the proof  of this theorem also for the case $ \mu \in \Pw \setminus D(\cF) $.
\end{proofoft} 

In the following corollary, we state some consequences of Theorem \ref{ThmConv} for the set $ \mathcal{B}^0 $.
\begin{corollary} \label{CorB0}
	\begin{enumerate}[(i)]
		\setlength\itemsep{-0.01em} 
		\item 
		$ \mathcal{B}^0 $ is closed,
		\item $ \mathcal{B}^0 \supset \{~\mu \in \Pw ~|~  \mu \text{ is symmetric, i.e.\ } \varsigma_\#\mu = \mu	~\} $, and 
		\item 	$ \mu^0  \in \partial  \mathcal{B}^0 $.
	\end{enumerate}
\end{corollary} 
\begin{proof}
	To show part \emph{(i)}, we simply use Proposition \ref{PropBasin} and that, by Theorem \ref{ThmConv}, $ \Pw = \mathcal{B}^- \cup \mathcal{B}^0 \cup \mathcal{B}^+ $. 
	Part \emph{(ii)} is a straightforward consequence of Theorem \ref{ThmConv} and Lemma \ref{LemBSymmetry}.
	Finally, to show part \emph{(iii)}, we use that by Proposition \ref{PropIntroSmallBasin}, $ \mu^{\varphi'(-\eta)} \in \mathcal{B}^- $ and $ \mu^{\varphi'(  \eta)} \in \mathcal{B}^+ $ for all $ \eta >0 $, and that 
	$ \lim_{\eta \downarrow  0 } W_2(\mu^{\varphi'( -\eta)}, \mu^0) ~=~ \lim_{\eta \downarrow  0 } W_2(\mu^{\varphi'( \eta)}, \mu^0)~=~0$.
	%
\end{proof} 

\section{Some comments on the assumptions in this paper.}\label{SecAssu}

In this chapter we briefly discuss
the assumptions we make in this paper. 

We first discuss 
Assumption \ref{AssBasin}. 
Assumption \ref{AssBasin}  (1) ensures that $ z \mapsto \eee^{-\Psi(z)  } $ is integrable and that
$ 		\Psi'' \geq  \tilde{\lambda}  $   for some $ \tilde{\lambda} \in \R $.
The latter condition implies Lemma \ref{LemProperties} (ii), which is an essential ingredient in order to apply the Wasserstein gradient flow theory for the functional $ \cF $; see \cite[Section 10.4]{ambgigsav}. 
Assumption \ref{AssBasin}  (2) implies  that the absolute moments of order $ 2+\epsilon $ of the McKean-Vlasov evolution are uniformly bounded; see   \eqref{EqBasinLBF} and  \eqref{EqComp1}. 
This uniform boundedness in turn implies some compactness property, which is an essential ingredient for the proofs;  see, e.g., the proof of Lemma \ref{LemCompactness}. 
Assumption \ref{AssBasin}  (3)   implies some symmetry properties of $ \cF $ that simplify our analysis considerably; see the  Lemmas \ref{LemRelationFH} and \ref{LemBSymmetry}. 
We   believe that   Assumptions \ref{AssBasin}  (3) is not essential and can be circumvented. 
Assumption \ref{AssBasin}  (4)  is the main reason
why the system \eqref{EqSDEBasin} admits exactly the three stationary states $\mu^-,\mu^0, \mu^+  $.
Indeed, as it can be seen in \cite[3.4]{BaMe2019}, if Assumption \ref{AssBasin}  (4) is not true, then the  macroscopic Hamiltonian  $ \bar{H} $  admits at most one critical point.
This implies that the system   \eqref{EqSDEBasin} admits at most one stationary state, since the critical points of $ \bar{H} $ determine the stationary states of the system \eqref{EqSDEBasin}; see     the proof of Lemma \ref{LemStationary}.   
Assumption \ref{AssBasin}  (5) is a technical assumption taken from \cite[Chapter 3]{BaMe2019}, where it is used in the proof of the fact that  $ \bar{H} $  admits exactly three critical points. 
We   believe that also  Assumptions \ref{AssBasin}  (5) is not essential and can be circumvented. The latter is left for future research.

We finally note that the Wasserstein gradient flow theory holds in a much more generality than it is used here. Hence, one may wonder if  the  results of this paper can be extended to more general settings. 
Unfortunately, our arguments rely on the facts that the system \eqref{EqSDEBasin} is one-dimensional and that the interaction energy in \eqref{EqDefiF} is quadratic. 
It is also left for future research to generalize the results of this paper to multi-dimensional settings and for more general interaction energies.

\renewcommand{\baselinestretch}{0.89}\normalsize
\bibliography{/Users/kaveh/Dropbox/Reference_File/TEX-Bib.bib}

\textbf{Acknowledgement.}
The author would like to thank Anton Bovier, Matthias Erbar and Andr\'e Schlichting for numerous useful discussions.  
Moreover, he would like to thank the anonymous referees for reading the paper with
great care and for their valuable comments. 
\end{document}